\documentclass[a4paper,12pt]{amsart}
\usepackage[utf8]{inputenc}
\usepackage[english]{babel}

\usepackage{amssymb, amsthm, relsize}
\usepackage[alphabetic]{amsrefs}

\usepackage{etoolbox}
\apptocmd{\sloppy}{\hbadness 10000\relax}{}{}

\newtheorem{theorem}{Theorem}[section]
\newtheorem{lemma}[theorem]{Lemma}
\newtheorem{proposition}[theorem]{Proposition}

\theoremstyle{definition}
\newtheorem{definition}[theorem]{Definition}

\theoremstyle{remark}
\newtheorem{remark}[theorem]{Remark}

\newcommand{\CC}{\mathbb{C}}
\newcommand{\NN}{\mathbb{N}}

\newcommand{\id}{\mathop{\mathrm{id}}}
\newcommand{\diff}[2]{\mathrm{d}_{ #1 } #2 }
\newcommand{\diag}[1]{\mathop{\mathrm{diag}\left(#1\right)} }
\newcommand{\ball}[1]{ B_{ #1 } }
\newcommand{\poly}[1]{ \Delta_{ #1 } }
\newcommand{\norm}[1]{ \left\lVert#1\right\rVert }
\newcommand{\abs}[1]{ \left\lvert#1\right\rvert }
\newcommand{\mat}[3]{ #1^{#2\times #3} }
\newcommand{\homogp}[1]{\mathcal{P}^{#1}}
\newcommand{\specialp}[2]{\mathcal{X}^{#1}_{#2}}
\newcommand{\landau}[1]{O\!\left(\abs{\cdot}^{#1}\right)}
\newcommand{\multii}[1]{Q\!\left(#1\right)}
\newcommand{\multiisq}[1]{Q^2\!\left(#1\right)}

\author{Rafael~B.~Andrist}
\address{Rafael~B.~Andrist \\ School of Mathematics and Natural Sciences \\ University of Wuppertal \\ Germany}
\email{rafael.andrist@math.uni-wuppertal.de}
\author{Gerrit~Maus}
\address{Gerrit~Maus \\ School of Mathematics and Natural Sciences \\ University of Wuppertal \\ Germany}
\email{maus@math.uni-wuppertal.de}

\title[Attracting sequences of holomorphic automorphisms]{Attracting sequences of holomorphic automorphisms that agree to a certain order}
\subjclass{32H50, 32H02}
\keywords{basins of attraction, complex dynamics, Fatou--Bieberbach domain}

\begin{document}

\begin{abstract}
The basin of attraction of a uniformly attracting sequence of holomorphic automorphisms that agree to a certain order in the common fixed point, is biholomorphic to $\mathbb{C}^n$. We also give sufficient estimates how large this order has to be.
\end{abstract}

\maketitle

\section{Introduction}

The systematic study of basins of attraction of holomorphic automorphisms goes back to works of Sternberg \cite{Sternberg} and Dixon and Esterle \cite{DixonEsterle}. The first complete result\footnote{Rosay and Rudin \cite{RosayRudin} remarked that the result of Dixon and Esterle \cite{DixonEsterle} relies on a statement of Reich \cite{Reich} which in turn had a gap in its proof. The earlier result of Sternberg \cite{Sternberg} deals only with an automorphism whose differential in the attracting fixed point is diagonal and has no special elements -- for the definition of special elements, see Section \ref{sc:proofs}.} was obtained by Rosay and Rudin \cite{RosayRudin} who showed that the basin of attraction of a holomorphic automorphism with an attracting fixed point is always biholomorphic to $\CC^n$. The question can be generalized to sequences of holomorphic automorphisms $f_j \colon \CC^n \to \CC^n$ with a common attracting fixed point $z_0 \in \CC^n$.

\begin{definition}
Let $f_j \colon \CC^n \to \CC^n$, $j \in \NN$, be a sequence of holomorphic self-maps. Their \emph{basin of attraction in $z_0 \in \CC^n$} is defined to be
\[
\Omega^{z_0}_{f_j} := \left\{ z \in \CC^n \,:\, \lim_{j \to \infty} f_j \circ f_{j-1} \circ \dots \circ f_1(z) = z_0 \right\}
\]
\end{definition}

A counterexample by Forn{\ae}ss \cite{Fornaess} shows that this basin of attraction of holomorphic automorphisms with a common fixed point $z_0$ does in general not need to be biholomorphic to $\CC^n$.

\smallskip

Therefore the question whether this basin of attraction is biholomorphic to $\CC^n$ is usually considered for automorphisms that satisfy the following uniform attraction property:

\begin{definition}
We call a sequence of holomorphic automorphisms $f_j \colon \CC^n \to \CC^n$, $j \in \NN$, \emph{uniformly attracting} in a point $z_0 \in \CC^n$, if there exist real numbers $0 < r, s < 1$ and $\delta > 0$ such that
\begin{equation*}
\forall j \in \NN \, \forall z \in \ball{\delta} \quad s \| z - z_0 \| < \| f_j(z) - z_0 \| < r \| z - z_0 \|
\end{equation*}
where $\ball{\delta} := \{ z \in \CC^n \,:\, \|z\| < \delta \}$.
\end{definition}

It was shown by Forn{\ae}ss and Stens{\o}nes \cite{FornaessStensones} that if $\Omega_{f_j}^{z_0}$ is biholomorphic to $\CC^n$ for any sequence of uniformly attracting holomorphic automorphisms, then this would give a positive answer to the stable manifold conjecture of Bedford. Their result has drawn a lot of interest and several positive partial results have been obtained so far. In particular we want to mention a result of Wold \cite{WoldFB}*{Theorem 4} that has been generalized by Sabiini and then further improved by Peters and Smit \cite{PetersSmit}.

\begin{theorem}\cites{SabiiniThesis,Sabiini}
\label{th:sab}
Let $0 < r, s < 1$, $0 < \delta$, and let $p \in \NN$ such that $r^p < s$.
Then for any uniformly attracting sequence of holomorphic automorphisms $f_j \colon \CC^n \to \CC^n$ with
\begin{equation*}
\forall j \in \NN \, \forall z \in \ball{\delta} \quad s \| z - z_0 \| < \| f_j(z) - z_0 \| < r \| z - z_0 \|
\end{equation*}
such that
\begin{equation}
\label{eq:sab}
\frac{\partial^{|\alpha|} f_j}{\partial z^\alpha}(z_0) = 0 \text{ for all multi-indices } \alpha \in \NN_0^n \text{ with } 2 \leq |\alpha| \leq p-1
\end{equation}
the basin of attraction is biholomorphic to $\CC^n$.
\end{theorem}
This contains the result of Wold \cite{WoldFB}*{Theorem 4} for $p=2$ where the condition \eqref{eq:sab} is empty. Recently, this condition was further improved in dimension $n=2$ by Peters and Smit \cite{PetersSmit} using the method of so-called adaptive trains.

\smallskip

Another positive result was obtained by Peters \cite{Peters} when all the uniformly attracting automorphisms $f_j$ are uniformly close to a given automorphism:
\begin{theorem}[\cite{Peters}]
\label{th:peters}
Given a holomorphic automorphism $f_1 \colon \CC^n \to \CC^n$ with $0$ as attractive fixed point there exists $\varepsilon > 0$ such that for any sequence of holomorphic automorphisms $f_j \colon \CC^n \to \CC^n$ fixing $0$ and satisfying $\| f_1 - f_j \|_{B_1} < \varepsilon$, the basin of attraction is biholomorphic to $\CC^n$.
\end{theorem}

In our paper we want to consider the situation when the higher partial derivatives in the common fixed point do not necessarily vanish, but instead agree up to a certain order $q$. 

\pagebreak
\begin{theorem}
\label{th:main}
Given $0 < r, s < 1$, $0 < \delta$, and a holomorphic automorphism $f_1 \colon \CC^n \to \CC^n$ there exists a number $q \in \NN$ such that for any uniformly attracting sequence of holomorphic automorphisms $f_j \colon \CC^n \to \CC^n$ with
\begin{equation*}
\forall j \in \NN \, \forall z \in \ball{\delta} \quad s \| z - z_0 \| < \| f_j(z) - z_0 \| < r \| z - z_0 \|
\end{equation*}
that agree in the fixed point $z_0$ modulo terms of order $q$, i.e.\
\begin{equation}
\label{eq:der_equ}
	\frac{\partial^{|\alpha|} f_1}{\partial z^\alpha}(z_0) = \frac{\partial^{|\alpha|} f_j}{\partial z^\alpha}(z_0)
\end{equation}
for all multi-indices $\alpha \in \NN_0^n$ with $1 \leq |\alpha| \leq q - 1$, the basin of attraction is biholomorphic to $\CC^n$.

Let $p \in \NN$ be such that $r^p < s$. Then we have the following estimates for $q$:
\begin{enumerate}
\item If each of the derivatives $\diff{z_0}{f_j}$, $j \in \NN$, is normal, then
\begin{equation*}
q \leq 
\frac{\ln\left(
			\mathop{\mathlarger{\mathlarger{\mathlarger\sum}}}\limits_{i=1}^{\left(p-1\right)^{\left(n-1\right)^2}}
			\!\!\!\!\!\!\multii{i}\left(\frac{n!}{s^{n}}
					\left(
						1+r\sum_{m=2}^{p-1}\limits n^2\multiisq{m}\left(
						\frac{\sqrt{n}}{\min\left\{1,\delta\right\}}
						\right)^m
					\right)^{n-1}
			\right)^i
			\right)}{\ln\frac{1}{r}}
		+1
\end{equation*}
independent of $f_1$,
\item In case of dimension $n=2$ and if each of the derivatives $\diff{z_0}{f_j}$, $j \in \NN$, is normal, then
\begin{equation*}
q\le\frac{\ln\left(
			\mathop{\mathlarger{\mathlarger{\mathlarger\sum}}}_{i=1}^{p-1}\limits
			(i+1)\left(\frac{2}{s^{2}}
					\left(
						1+2rp \left(
						\frac{\sqrt{2}}{\min\left\{1,\delta\right\}}
						\right)^{p-1}
					\right)
			\right)^i
			\right)}{\ln\frac{1}{r}}
		+1 \in O(p^2)
\end{equation*}
independent of $f_1$,
\end{enumerate}
where \(\multii{m}\) denotes the number of
multi-indices in \(n\) variables of order \(m\).
\end{theorem}

\section{The Rosay--Rudin framework}

In this section we state and prove the key proposition which goes back to Rosay and Rudin \cite{RosayRudin}*{Appendix} for the basin of attraction of a single automorphism. Several special cases of this proposition have been used in the literature, but to the authors' knowledge, it has never been stated as a separate result in full generality. The rather technical assumptions will become clear in the applications. As an immediate corollary we will obtain the aforementioned result of Sabiini, Theorem \ref{th:sab}.

\pagebreak
We will use the following convenient notation as in \cite{WoldFB}:
\begin{definition}
Let $f_j \colon \CC^n \to \CC^n, j \in \NN,$ be a sequence of holomorphic automorphisms. Then we set
	\begin{equation*}
	f_{j, k} := \begin{cases}
		                 	f_k\circ f_{k-1}\circ \dotsb \circ f_{j+1}
		                 	\circ f_j & \text{if } j \le k,\\
		                 	\id & \text{else}
		        \end{cases}
	\end{equation*}
	
	and
	\begin{equation*}
		f_{j, k}^{-1} := \left(f_{j, k}\right)^{-1}.
	\end{equation*}
\end{definition}

\begin{proposition}
\label{pp:rr_biholo}
Let $f_j \colon \CC^n \to \CC^n$, $j \in \NN$, be holomorphic automorphisms fixing $0$.
Assume there exist $\delta > 0$, $1 > r > 0$ and $K \in \NN$ such that
for all	$z \in \ball{\delta}$ and for all $j, k \in \NN$ with $k \geq j + K - 1$ the following holds:
\begin{equation}
\label{eq:rra_rra_bound}
\|f_{j,k}(z) \| \leq r^{k-j+1} \| z \|
\end{equation}
Moreover we assume that for each $j \in \NN$ there exist holomorphic automorphisms $G_j \colon \CC^n \to \CC^n$ and holomorphic self-maps $T_j \colon \CC^n \to \CC^n$ that satisfy the following:
\begin{equation}
\label{eq:rra_rra_maps_0}
T_j(0) = 0, \quad G_j(0) = 0
\end{equation}
\begin{equation}
\label{eq:rra_rra_maps_diff}
\forall j \in K \NN_0 \;:\quad \diff{0}{G_{j+1,j+K}} = f_{j+1,j+K}, \quad \diff{0}{T_{j+K}} = \id,
\end{equation}
\begin{equation}
\label{eq:rra_rra_maps_conv}
G_{1,j} \rightrightarrows 0 \text{ on compacts for } j \to \infty
\end{equation}
\begin{equation}
\label{eq:rra_rra_maps_bound}
\exists b > 0 \forall z \in \ball{\delta} \forall j \in K\NN \,:\, \|T_j(z)\| \leq b \|z\|
\end{equation}
We further assume that there exists an open neighborhood $W$ of $0 \in \CC^n$ such that 
\begin{equation}
\label{eq:rra_rra_maps_inj}
\forall j \in K\NN \,:\, T_j|_W \text{ is one-to-one}
\end{equation}
and that there exist $\rho > 0$, $\gamma > 0$ and $a > 0$ such that
\begin{equation}
\label{eq:rra_rra_maps_sp}
\begin{split}
\forall z, z^\prime \in \ball{\rho} \; \forall j, k \in K\NN, k > j :\\
\| G_{j+1,k}^{-1}(z) -G_{j+1,k}^{-1}(z^\prime) \| \leq a \gamma^{k-j} \| z - z^\prime \|
\end{split}
\end{equation}
and that there exist $c > 0$ and $q \in \NN$ with $r^q \gamma < 1$ and 
\begin{equation}
\label{eq:rra_rra_maps_est}
\begin{split}
\forall z \in \ball{\delta} \; \forall j \in K\NN :\\
\| G_{j+1,j+K}^{-1} \circ T_{j+K} \circ f_{j+1, j+K}(z) - T_j(z) \| \leq c \| z \|^q
\end{split}
\end{equation}
Then the domain
\[
\Omega := \left\{ z \in \CC^n \,:\, \lim_{j \to \infty} f_{1,j}(z) = 0 \right\}
\]
is biholomorphic to $\CC^n$.
\end{proposition}

\begin{proof}
	We may choose \(\delta>0\) arbitrarily small.
	Hence w.l.o.g. let
	\begin{equation}
	\label{eq:rra_rra_obda}
		b\delta<\rho.
	\end{equation}
	
	By \eqref{eq:rra_rra_bound} we get
	\begin{equation*}
		\forall\,j\ge K\ \forall\,z\in \ball{\delta}:
		\abs{f_{1,j}(z)}\le r^j \abs{z} < r^j\delta.
	\end{equation*}
	Hence we have uniform convergence
	\begin{equation}
	\label{eq:rra_rra_uni}
		f_{1,j}\vert_{B_\delta}\rightrightarrows 0
	\end{equation}
	and it follows that
	\begin{equation*}
		\Omega\subset\bigcup_{j\in\mathbb{N}_0}f_{1,j}^{-1}\left(
		\ball{\delta}\right).
	\end{equation*}
	Conversely, let \(j\in\mathbb{N}_0\) und \(z\in f_{1,j}^{-1}\left(
	\ball{\delta}\right)\). Then we have \(f_{1,j}(z)\in B_\delta\)
	and hence
	for \(k\ge j + K\)
	\begin{equation*}
		\norm{f_{1,k}(z)}=\norm{f_{j+1,k}\left(f_{1,j}(z)\right)}
		\le r^{k-j}\norm{f_{1,j}(z)}
		\xrightarrow{k\to\infty}0.
	\end{equation*}
	Altogether we obtain
	\begin{equation}
	\label{eq:rra_rra_union}
		\Omega=\bigcup_{j\in\mathbb{N}_0}f_{1,j}\left(
		\ball{\delta}\right).
	\end{equation}
	In particular we have that \(\Omega\) is an open and connected subset of
	\(\CC^n\).
	With \eqref{eq:rra_rra_uni} it follows in addition that
	\begin{equation}
	\label{eq:rra_rra_compact}
		f_{1,j}\rightrightarrows 0 \text{ on compacts in } \Omega.
	\end{equation}
	
	For \(l\in\mathbb{N}\) we define the sequence \(\left(\psi_j^l
	\right)_{j\ge l+K}\) of maps
	\begin{equation*}
		\psi_j^l\colon \ball{\delta}\to\mathbb{C}^n
	\end{equation*}
	by
	\begin{equation*}
		\psi_j^l:=G_{l+1,j}^{-1}\circ T_j\circ 
		f_{l+1,j}.
	\end{equation*}
	Now let \(l,j\in K\mathbb{N}\) with \(j\ge l+K\).
	From \eqref{eq:rra_rra_bound} and \eqref{eq:rra_rra_maps_bound}
	we get for \(z\in\ball{\delta}\) that
	\begin{equation*}
		\norm{f_{l+1,j}(z)}\le r^{j-l}\delta,
	\end{equation*}
	\begin{equation*}
		\norm{T_j\circ f_{l+1,j}(z)}\le b r^{j-l}\delta
	\end{equation*}
	and
	\begin{equation*}
		\norm{T_{j+K}\circ f_{j+1,j+K}(z)}
		\le b r^K\norm{z}.
	\end{equation*}
	And from \eqref{eq:rra_rra_maps_0} and \eqref{eq:rra_rra_maps_sp}
	we get for \(z\in \ball{\rho}\) that
	\begin{equation*}
		\norm{G_{j+1,j+K}^{-1}(z)}\le a\gamma^K\norm{z}.
	\end{equation*}
	Altogether we have for
	\begin{align*}
		\tilde{z}&:=G_{j+1,j+K}^{-1}\circ T_{j+K}\circ
		f_{j+1,j+K}\left(
		f_{l+1,j}(z)\right)\\
		\hat{z}&:=T_j\left(f_{l+1,j}(z)\right),
	\end{align*}
	that
	\begin{equation*}
		\begin{split}
			\exists\,J\in K\mathbb{N}\ \forall\,l,j\in K\mathbb{N}
			,\,j\ge l + J\ \forall\,z\in\ball{\delta}:
			\norm{\tilde{z}}\le\rho,
			\norm{\hat{z}}
			\le\rho.
		\end{split}
	\end{equation*}
	Hence together with \eqref{eq:rra_rra_bound}, \eqref{eq:rra_rra_maps_sp} and
	\eqref{eq:rra_rra_maps_est} we obtain
	\begin{equation}
	\label{eq:rra_rra_cs}
		\begin{split}
			&\forall\,l,j\in K \mathbb{N},\,j\ge l+J\ \forall\,
			z\in\ball{\delta}:\\
			\norm{\psi_{j+K}^l(x)-\psi_j^l(x)}&=\norm{G_{l+1,j}^{-1}\left(
			\tilde{z}\right)- G_{l+1,j}^{-1}\left(\hat{z}\right)}\\
			&\le a\gamma^{j-l}\norm{\tilde{z}-\hat{z}}\\
			&\le a\gamma^{j-l}c\norm{f_{l+1,j}(z)}^q\\
			&\le a\gamma^{j-l}c\left(r^{j-l}\delta\right)^q\\
			&= ac\delta^q\left(r^q\gamma\right)^{j-l}.			
		\end{split}
	\end{equation}
	It follows that the subsequences \(\left(\psi_j^l\right)_{j\in K
	\mathbb{N}}\) converge uniformly (on \(\ball{\delta}\))
	to maps
	\(\psi^l\colon \ball{\delta}\to\mathbb{C}^n\).
	
	From \eqref{eq:rra_rra_cs} we also get (together with 
	\eqref{eq:rra_rra_obda}) that
	\begin{align*}
		\norm{\psi^l}_{\ball{\delta}}&=
		\norm{\psi^l-\psi_{l+J}^l+\psi_{l+J}^l}_{\ball{\delta}}\\
		&\le \sum_{\substack{j\in K\mathbb{N}\\j\ge l + J}}
		\norm{\psi_{j+K}^l-\psi_j^l}_{\ball{\delta}}
		+\norm{\psi_{l+J}^l}_{\ball{\delta}}\\
		&<\sum_{\substack{j\in K\mathbb{N}\\j\ge l + J}}
		ac\delta^q\left(r^q\gamma\right)^{j-l}+a\gamma^Jbr^J\delta\\
		&=\sum_{\substack{j\in K\mathbb{N}\\j\ge J}}
		ac\delta^q \left(r^q\gamma\right)^{j}+a\gamma^Jbr^J\delta<\infty.
	\end{align*}
	This estimate does not depend on \(l\). Hence we obtain
	\begin{equation}
	\label{eq:rra_rra_R}
		\exists\,R>0\ \forall\,l\in K\mathbb{N}:\quad
		\norm{\psi^l}_{\varphi^{-1}\left(B_\delta\right)}<R.
	\end{equation}
	
	Now we consider the sequence \(\left(\psi_j\right)_{j\in K\mathbb{N}}\)
	of holomorphic maps \linebreak \(\psi_j\colon \Omega \to \mathbb{C}^n\)
	defined by
	\begin{equation*}
		\psi_j := G_{1,j}^{-1}\circ T_j\circ f_{1,j}.
	\end{equation*}
	Let \(E\subset\Omega\) be a compact set. From
	\eqref{eq:rra_rra_compact} we get
	\begin{equation*}
		\exists\,l\in K\mathbb{N}:
		f_{1,l}\left(E\right)\subset \ball{\delta}.
	\end{equation*}
	Hence we have for \(j\in K \mathbb{N}\)
	with \(j\ge l+K\) that
	\begin{equation}
	\label{eq:rra_rra_comp}
		\psi_j\vert_E = G_{1,l}^{-1}\circ
		\psi_j^l\circ f_{1,l}\vert_E
	\end{equation}
	Therefore the uniform convergence of
	\(\left(\psi_j^l\right)_{j\in K\mathbb{N}}\) on
	\( \ball{\delta}\) implies the uniform convergence of
	\(\left(\psi_j\right)_{j\in K\mathbb{N}}\)
	on \(E\). Hence \(\left(\psi_j\right)_{j\in K\mathbb{N}}\)
	converges uniformly on compacts of \(\Omega\)
	to a holomorphic map \(\psi\colon
	\Omega\to\mathbb{C}^n\).
	
	From \eqref{eq:rra_rra_maps_diff} we get
	\begin{equation}
	\label{eq:rra_rra_detdiff}
			\det \diff{0}{\psi} =\det\id = 1.
	\end{equation}
	Let \(x,y\in\Omega\) with \(\psi(x)=\psi(y)\).
	There exists a relatively compact, open and connected set
	\(E\subset\Omega\) with \(0,x,y\in E\).
	By
	\eqref{eq:rra_rra_compact} it follows that
	\begin{equation*}
		\exists\,L\in K\mathbb{N}\ \forall\,j\in K\mathbb{N},\,j\ge L:
		f_{1,j}\left(E\right)\subset W.
	\end{equation*}
	By \eqref{eq:rra_rra_maps_inj} we then know \(\psi_j\)
	to be one-to-one on \(E\) for such \(j\).
	Then \eqref{eq:rra_rra_detdiff} implies that \(\psi\)
	is one-to-one on \(E\). Hence \(x=y\) holds and we have shown that
	\(\psi\) is one-to-one.
	
	For \(l\in K\mathbb{N}\) we have by
	\eqref{eq:rra_rra_bound} that
	\begin{equation*}
		\exists\,J\in K\mathbb{N}\ \forall\,j\in K\mathbb{N},\,j\ge J:
		f_{l+1,j}\left(\ball{\delta}\right)\subset W.
	\end{equation*}
	Hence \(\psi_j^l\) is one-to-one for such \(j\).
	Like \eqref{eq:rra_rra_detdiff}
	we also have
	\begin{equation}
	\label{eq:rra_rra_detdiffl}
		\det\diff{0}{\psi^{l}} \ne 0.
	\end{equation}
	Hence \(\psi^l\) is one-to-one and open.
	By \eqref{eq:rra_rra_R} and Montel's Theorem there exists
	a subsequence of \(\left(\psi^l\right)_{l\in K\mathbb{N}}\)
	converging to a holomorphic map
	\(\hat{\psi}\colon\ball{\delta}\to\mathbb{C}^n\).
	From \eqref{eq:rra_rra_detdiffl} it follows
	\(\det\diff{0}{\hat{\psi}}\ne 0\).
	Hence \(\hat{\psi}\) is also one-to-one.
	Altogether we have for \(0<\tilde{\delta}<\hat{\delta}<\delta\)
	that
	\begin{equation*}
		\exists\,L\in K\mathbb{N}\ \forall\,l\in K\mathbb{N},\,l\ge L:
		\hat{\psi}\left(\overline{\ball{\tilde{\delta}}}\right)
		\subset \psi^l\left(\overline{B_{\hat{\delta}}}\right).
	\end{equation*}
	Together with \(\hat{\psi}(0)=0\) we obtain
	\begin{equation}
	\label{eq:rra_rra_eps}
		\exists\,\varepsilon>0\ \forall\,l\in K\mathbb{N},\,l\ge L:
		\ball{\varepsilon}\subset  \psi^l\left(
		\overline{B_{\hat{\delta}}}\right).
	\end{equation}
	
	Let \(M>0\) be arbitrarily large.
	By \eqref{eq:rra_rra_maps_conv} there exists
	\(l\in K\mathbb{N}\), \(l\ge L\), s.t.
	\(G_{1,l}\left(\ball{M}\right)\subset
	B_\varepsilon\). Define a compact set
	\(E:=f_{1,l}^{-1}\left(
	\overline{B_{\hat{\delta}}}\right)\).
	Then \eqref{eq:rra_rra_union} implies
	\begin{equation*}
		E\subset f_{1,l}^{-1}\left(\ball{\delta}
		\right)\subset\Omega.
	\end{equation*}
	Because of \(f_{1,l}\left(E\right)\subset\left(
	\ball{\delta}\right)\) we have by
	\eqref{eq:rra_rra_comp} that
	\begin{equation*}
		\psi\vert_E = G_{1,l}^{-1}\circ
		\psi^l\circ f_{1,l}\vert_E.
	\end{equation*}
	Together with \eqref{eq:rra_rra_eps} we finally obtain
	\begin{align*}
		\ball{M}\subset G_{1,l}^{-1}\left(\ball{\varepsilon}\right)
		\subset G_{1,l}^{-1}\left(\psi^l\left(\overline{\ball{\hat{\delta}}}\right)\right)
		&= G_{1,l}^{-1}\left(\psi^l\left(f_{1,l}
		\left(E\right)\right)\right)\\
		&=\psi\left(E\right)\\
		&\subset\psi\left(\Omega\right).
	\end{align*}
	Hence the image of \(\Omega\) under \(\psi\) is the whole of \(\CC^n\).
\end{proof}

\begin{proof}[Proof of Theorem \ref{th:sab}]
	We use Proposition \ref{pp:rr_biholo}.
	
	Let \(A_{j}:=\diff{0}{f_{j}}\).
	By \eqref{eq:sab} we then have
	\begin{equation}
	\label{eq:rra_sabiini_maps_est}
		A_{j}^{-1}\circ f_{j}-\id \in
		\landau{p}.
	\end{equation}
	We define
	\begin{align*}
		T_j &:= \id\\
		G_j &:= A_j
	\end{align*}
	Clearly, the assumptions \eqref{eq:rra_rra_maps_0},
	\eqref{eq:rra_rra_maps_diff},
	\eqref{eq:rra_rra_maps_bound} and
	\eqref{eq:rra_rra_maps_inj}
	of Proposition \ref{pp:rr_biholo} are
	fulfilled.
	
	We have \(\norm{G_j}\le r\) and it is easy to see that
	\(\norm{G_j^{-1}}\le\frac{1}{s}\). Hence
	\eqref{eq:rra_rra_maps_conv} holds
	and \eqref{eq:rra_rra_maps_sp} holds for an arbitrary \(\rho>0\), \(a=1\)
	and \(\gamma = \frac{1}{s}\).
	All \(G_j^{-1}\) and
	all \(f_j\) are uniformly bounded (on \(\ball{\delta}\)). Hence
	\eqref{eq:rra_sabiini_maps_est} implies
	\eqref{eq:rra_rra_maps_est}.
	
	Together with
	\(r^p\gamma=\frac{r^p}{s}<1\)
	the theorem now follows from Proposition \ref{pp:rr_biholo} using \(q=p\).
\end{proof}


\section{Proofs}
\label{sc:proofs}
In order to prove Theorem \ref{th:main} using Proposition \ref{pp:rr_biholo} we will need the following lemmas with quantitative estimates.
Therefore we will need some terminology introduced by Rosay--Rudin \cite{RosayRudin}*{Appendix}.

Let \(c_1, \dotsc, c_n\in\mathbb{C}\) and for \(\nu\in\left\{
2, \dotsc, n\right\}\) let
\(h_\nu\colon\mathbb{C}^{\nu-1}\to\CC\)
be holomorphic maps with \(h_\nu(0)=0\).
A holomorphic map \(G\colon\mathbb{C}^n\to\mathbb{C}^n\), 
\(G=(g_1, \dotsc, g_n)\), is called
\emph{lower triangular}, if
\begin{align*}
	g_1(z_1, \dotsc, z_n) &= c_1z_1\\
	g_2(z_1, \dotsc, z_n) &= c_2z_2+h_2(z_1)\\
	                      &\,\ \vdots \\
	g_n(z_1, \dotsc, z_n) &= c_nz_n+h_n(z_1, \dotsc, z_{n-1}).
\end{align*}
It is called \emph{polynomial lower triangular}, if
all \(h_\nu\) are polynomial. In this case we define
\begin{equation*}
	\deg G := \max_{\nu\in\left\{1,\dotsc,n\right\}} \deg g_\nu.
\end{equation*}

Those elements \(c_\nu\) are called \emph{diagonal elements}.
It is easy to see that \(G\) is a holomorphic automorphism
if and only if no \(c_\nu\) vanishes.

For \(m\in\NN\), \(m\ge 2\), let \(\homogp{m}\) denote
the vector space of all holomorphic maps \(h\colon \CC^n\to\CC^n\)
whose components consist of homogeneous polynomials in
\(n\) variables of order \(m\).
Let \(\lambda_1,\dotsc,\lambda_n\) denote the eigenvalues of
a linear map \(A\) s.t. \(0<\abs{\lambda_n}\le\dotsb\le\abs{\lambda_1}<1\).
Clearly, the maps of the form \(h(z)=(0,\dotsc,0,z^\alpha,0,\dotsc,0)\),
\(\alpha=(\alpha_1,\dotsc,\alpha_n)\), \(\abs{\alpha}=m\),
provide a basis of \(\homogp{m}\).
Such an \(h\) with all components but the \(j\)th vanishing
is called \emph{special (with respect to \(A\))},
if \(\alpha_j=\dotsb=\alpha_n=0\) and
\(\lambda_j = \lambda_1^{\alpha_1}\dotsm \lambda_{j-1}^{\alpha_{j-1}}\).
We denote the vector subspace of those special elements by \(\specialp{m}{A}\).

Rosay--Rudin \cite{RosayRudin}*{Appendix} observed
the following: If we have \(\abs{\lambda_1}^p<\abs{\lambda_n}\)
for some \(p\in\NN\), \(p\ge 2\), we get \(\specialp{m}{A}=0\)
for all \(m\ge p\).

\begin{lemma}
\label{th:rra_2}
	Let \(A\in\mat{\CC}{n}{n}\) be lower triangular
	with eigenvalues \(1>\abs{\lambda_1}\ge\dotsb \ge\abs{\lambda_n}>0\)
	down its main diagonal and let
	\(m\in\mathbb{N}\), \(m \ge 2\). Then for every
	\(R\in\homogp{m}\) there exist \(X\in\specialp{m}{A}\)
	and \(H\in\homogp{m}\) s.t. \(R=X+\Gamma_AH\) where
	\(\Gamma_A\colon\homogp{m}\to \homogp{m}\)
	is the commutator map \(\Gamma_AH:= A\circ H-H\circ A\).
	If \(A\) is diagonal, then
	\(X\) can be chosen to satisfy
	\(\norm{X}_{\poly{1}}\le nQ(m)\norm{R}_{\poly{1}}\).
\end{lemma}

\begin{proof}
	The proof of the general case (without estimate) is due to
	Rosay and Rudin \cite{RosayRudin}*{Appendix, Lemma 2}.
	They first consider our special case
	\(A=\diag{\lambda_1,\dotsc,\lambda_n}\).
	In this case \(\Gamma_A\) is an invertible linear operator
	on the space of non-special elements of \(\homogp{m}\).
	Let \(P\colon \homogp{m}\to\specialp{m}{A}\)
	be the projection onto \(\specialp{m}{A}\) and set
	\(X:= PR\) and \(H:=\left(\Gamma_A\right)^{-1}\left(R-PR\right)\).
	Then the estimate follows from \(\norm{P}\le n Q(m)\).
\end{proof}

\begin{remark}
\label{th:rem_special}
	In case of dimension \(n=2\) the vector subspace
	\(\mathcal{X}_A^m\) is at most one-dimensional.
	In this case it is spanned by
	\begin{equation*}
		\left(z_1,z_2\right) \mapsto
		\left(0,z_1^m\right).
	\end{equation*}
	For
	\begin{equation*}
		R\left(z_1,z_2\right)=
		\left(\sum_{i=0}^m a_{1,i}z_1^iz_2^{m-i},
		\sum_{i=0}^m a_{2,i}z_1^iz_2^{m-i}\right),
	\end{equation*}
	we then have:
	\begin{equation*}
	\begin{split}
		\norm{X}_{\Delta_1}&\le
		\sup_{\left(z_1,z_2\right)\in\Delta_1}\norm{a_{2,m}
		\begin{pmatrix}
			0\\z_1^m
		\end{pmatrix}}
		\le\sup_{\substack{\left(z_1,z_2\right)\in\Delta_1\\z_2=0}}
		\norm{\begin{pmatrix}
			\sum_{i=0}^m \limits a_{1,i}z_1^iz_2^{m-i}\\
			\sum_{i=0}^m \limits a_{2,i}z_1^iz_2^{m-i}
		\end{pmatrix}} \\
		&\le \norm{R}_{\Delta_1}
	\end{split}
	\end{equation*}
\end{remark}

\begin{lemma}
\label{th:rra_3}
	Let \(f\colon \CC^n\to\CC^n\) be a holomorphic map
	fixing 0 and
	\(A:=\diff{0}{f}\) with eigenvalues \(0<\abs{\lambda_i}<1\).
	
	Then there exist a unitary linear map \(S\),
	a polynomial lower triangular automorphism
	\(\widetilde{G}\) with \(\widetilde{G}(0)=0\) and
	\(\diff{0}{\widetilde{G}}=S^{-1}\circ A\circ S\) and
	polynomials \(T^m\colon\CC^n\to\CC^n\), \(m\in\NN\), \(m\ge 2\),
	with \(T^m(0)=0\) and \(\diff{0}{T^m}=\id\) s.t.
	\begin{equation*}
		S\circ \widetilde{G}^{-1}\circ S^{-1}\circ T^m\circ f-T^m\in
		\landau{m}.
	\end{equation*}
	
	\(S\) only depends on \(A\).
	For \(m>2\) all maps \(T^m\)
	only depend on the derivatives of \(f\) up to order \(m-1\).
	In addition we have \(T^2=\id\).
	If all derivatives of \(f\) up to order
	\(m-1\) are special respective \(A\), we have \(T^m=\id\).

	If for some \(p\in\mathbb{N}\), \(p\ge 2\), we have
	\(\mathcal{X}_A^m=0\) for all \(m\ge p\) then
	\(\widetilde{G}\) only depends on the derivatives of \(f\)
	up to order \(p-1\) and we have \(\deg \widetilde{G} \le p-1\).
	In particular we have \(\widetilde{G}=S^{-1}\circ A\circ S\)
	for \(p=2\).
	
	If \(A\) is normal and if we find \(0<\delta\le 1\) and \(0<s<r\) s.t.
	\begin{equation}
	\label{eq:rra_ass}
		\forall\,z\in B_\delta: s\norm{z}\le \norm{f(z)}\le r\norm{z},
	\end{equation}
	then we can write
	\begin{equation*}
		\widetilde{G}=S^{-1}\circ A\circ S+H
	\end{equation*}
	with
	\begin{equation*}
		\norm{S^{-1}\circ A^{-1}\circ S}\le\frac{1}{s}
	\end{equation*}
	and
	\begin{equation*}
		\forall\,z\in B_{\frac{\delta}{\sqrt{n}}}:
		\norm{H(z)}\le r\norm{z}\sum_{m=2}^{p-1}
		n^2\multiisq{m}\left(\frac{\sqrt{n}}{\delta}\right)^m.
	\end{equation*}
\end{lemma}

\begin{proof}
	We denote by \(\lambda_1,\dotsc,\lambda_n\)
	the eigenvalues of \(A\) s.t. \(0\le \abs{\lambda_n}\le
	\dotsc,\le \abs{\lambda_1}<1\) and
	first consider the special case that \(A\) is lower triangular
	with \(\lambda_1,\dotsc,\lambda_n\) down its main diagonal.
	The proof (without estimates) for this case is due to Rosay and Rudin
	\cite{RosayRudin}*{Appendix, Lemma 3}.
	
	We recall from their proof the inductive construction
	of those polynomials \(T^m\) and polynomial lower triangular automorphisms
	\(\widetilde{G}^m\) with 
	\(\widetilde{G}^m(0)=0\) and \( \diff{0}{\widetilde{G}^m}=A\) s.t.
	\begin{equation}
	\label{eq:rra_ih}
		T^m\circ f-\widetilde{G}^m\circ T^m\in\landau{m}.
	\end{equation}
	Let \(T^2 := \id\) and \(\widetilde{G}^2 := A\).
	If for some \(m\) the maps \(T^m\) and \(\widetilde{G}^m\)
	are constructed and \eqref{eq:rra_ih} holds, then there exists
	\(R^m\in\homogp{m}\) s.t.
	\begin{equation*}
		T^m\circ f-\widetilde{G}^m\circ T^m-R^m\in\landau{m+1}.
	\end{equation*}
	
	In the case that \(R^m\) is special, we set \(X^m:=R^m\)
	and \(H^m:=0\).
	Otherwise
	Lemma \ref{th:rra_2} gives us
	\(X^m\in\specialp{m}{A}\)
	and \(H^m\in\homogp{m}\) with
	\begin{equation*}
		R^m = X^m + A\circ H^m-H^m\circ A.
	\end{equation*}
	Note that if \(A\) is diagonal, we get (in both cases) the estimate
	\begin{equation}
	\label{eq:rra_3_le}
		\norm{X^m}_{\Delta_1}\le nQ(m)\norm{R^m}_{\Delta_1}.
	\end{equation}
	Now let \(\widetilde{G}^{m+1}:=\widetilde{G}^m+X^m\) and
	\(T^{m+1}:=T^m+H^m\circ T^m\).
	For \(m\) large enough we have \(\specialp{m}{A}=0\)
	and hence \(\widetilde{G}^m = \widetilde{G}^{m+1} =: G\).
	Those maps satisfy the desired properties with \(S=\id\).
	
	To prove the general case we find a unitary \(S\) s.t.
	\(\widetilde{A}:=S^{-1}\circ A\circ S\) meets the requirements
	of the special case above.
	For \(\widetilde{f}:=S^{-1}\circ f\circ S\) we will then find
	maps \(\widetilde{T}^m\)
	and \(\widetilde{G}\) s.t.
	\begin{equation*}
		\widetilde{G}^{-1}\circ \widetilde{T}^m\circ \widetilde{f}-
		\widetilde{T}^m\in \landau{m}.
	\end{equation*}
	With \(T^m:=S\circ
	\widetilde{T}^m\circ S^{-1}\) we can rewrite this to
	\begin{equation*}
		S\circ \widetilde{G}^{-1}\circ S^{-1}\circ T^m\circ f-T^m\in
		\landau{m}.
	\end{equation*}
	All formulated dependencies are obvious by construction.
	
	If \(A\) is normal, we can choose \(S\) s.t. \(\widetilde{A}\) is
	diagonal. The construction above yields
	\begin{equation}
	\label{eq:rra_3_est}
		\widetilde{G}=\widetilde{A}+\sum_{m=2}^{p-1}\widetilde{X}^m.
	\end{equation}
	By \eqref{eq:rra_ass} we have
	\(\norm{\widetilde{A}^{-1}}\le\frac{1}{s}\).
	All \(\widetilde{X}^m\) and \(\widetilde{R}^m\) are homogeneous polynomials.
	Hence from \eqref{eq:rra_3_le} follows
	\(\norm{\widetilde{X}^m}_{\Delta_\frac{\delta}{\sqrt{n}}}\le
	n\multii{m}\norm{\widetilde{R}^m}_{\Delta_\frac{\delta}{\sqrt{n}}}\).
	By the Cauchy integral formula we get
	\begin{equation*}
		\norm{\widetilde{R}^m}_{\Delta_\frac{\delta}{\sqrt{n}}}
		\le n\multii{m}\frac{\norm{F}_{\Delta_{\frac{\delta}{\sqrt{n}}}}}{
		\left(\frac{\delta}{\sqrt{n}}\right)^m}.
	\end{equation*}
	Together we have for all \(z\in B_{\frac{\delta}{\sqrt{n}}}\)
	\begin{align*}
		\norm{\widetilde{X}^m(z)}&=
		\norm{\widetilde{X}^m\left(\norm{z}\frac{\sqrt{n}}{\delta}
		\frac{z}{\norm{z}}\frac{\delta}{\sqrt{n}}\right)}\\
		&= \left(\norm{z}\frac{\sqrt{n}}{\delta}\right)^m
		\norm{\widetilde{X}^m}_{\Delta_{\frac{\delta}{\sqrt{n}}}}\\
		&\le \norm{z}\frac{\sqrt{n}}{\delta}n^2\multiisq{m}
		\left(\frac{\sqrt{n}}{\delta}\right)^m
		\norm{F}_{\Delta_{\frac{\delta}{\sqrt{n}}}}\\
		&\le \norm{z}n^2\multiisq{m}\left(\frac{\sqrt{n}}{\delta}\right)^mr.
	\end{align*}
	Hence for
	\begin{equation*}
		H:=\sum_{m=2}^{p-1}\widetilde{X}^m,
	\end{equation*}
	we have
	\begin{equation*}
		\forall\,z\in B_{\frac{\delta}{\sqrt{n}}}:
		\norm{H(z)}\le r\norm{z}\sum_{m=2}^{p-1}
		n^2 \multiisq{m}\left(\frac{\sqrt{n}}{\delta}\right)^m.\qedhere
	\end{equation*}
\end{proof}

\begin{remark}
\label{th:rem_est}
	In case of dimension \(n=2\)
	there can be at most one \(m\) for which
	\(\specialp{m}{A}\ne 0\). Hence at most one summand in \eqref{eq:rra_3_est}
	does not vanish. Together with
	Remark \ref{th:rem_special}
	we obtain a better estimate:
	\begin{equation*}
		\forall\,z\in B_{\frac{\delta}{\sqrt{2}}}:
		\norm{H(z)}\le 2r
		Q(p-1)\left(\frac{\sqrt{2}}{\delta}\right)^{p-1}\norm{z}.
	\end{equation*}
\end{remark}

\begin{lemma}
\label{th:rra_bounded}
	Let \(G\) be a lower triangular holomorphic automorphism with
	\(\deg G\le d\) for some \(d\in\NN\) and
	\begin{equation*}
		\forall\,z\in B_\rho:
		\norm{G(z)}\le C\norm{z}
	\end{equation*}
	for some \(0<\rho<1\) and \(C>0\).
	Then there exists \(\gamma > 0\) s.t.
	\begin{equation*}
		G_{1,k}\left(\poly{\frac{\rho}{\sqrt{n}}}\right)\subset
		\Delta_{\gamma^{k}}
	\end{equation*}
	where \(\poly{\delta}\) is the polydisc of radius \(\delta\) about \(0\).
	We may choose
	\begin{equation}
	\label{eq:rra_bounded_beta}
		\gamma:=\sum_{i=1}^{d^{n-1}}\multii{i}\left(\sqrt{n}\right)^iC^i.
	\end{equation}
\end{lemma}

\begin{proof}
	The proof without estimate is due to
	Rosay and Rudin \cite{RosayRudin}*{Appendix, Lemma 1}.
	They first observed
	\begin{equation}
	\label{eq:rra_bounded_deg}
		\forall\,k\in\mathbb{N}:\deg G_{1,k}\le d^{n-1}.
	\end{equation}
	
	We denote by \(\left(G\right)_\nu\) the \(\nu\)-th component of \(G\).
	Then we have
	\begin{equation}
	\label{eq:rra_bounded_bound}
		\forall\,z\in\Delta_{\frac{\rho}{\sqrt{n}}}:
		\abs{\left(G(z)\right)_\nu}\le C\rho\le \gamma.
	\end{equation}
	
	Now suppose that
	\begin{equation*}
		\forall\,z\in\Delta_{\frac{\rho}{\sqrt{n}}}:
		\abs{\left(G_{1,k}\right)_\nu}\le \gamma^{k}
	\end{equation*}
	for some \(k\in\NN\).
	By
	\eqref{eq:rra_bounded_deg} we have that
	\begin{equation*}
		\left(G_{1,k}(z)\right)_\nu = 
		\sum_{\substack{\alpha\in\mathbb{N}_0^n\\1\le\abs{\alpha}\le d^{n-1}}}
		a_{k,\nu,\alpha}z^\alpha
	\end{equation*}
	and by the Cauchy integral formula we get
	\begin{equation*}
		\abs{a_{k,\nu,\alpha}} 
		\le \frac{\gamma^{k}}{\rho^{\abs{\alpha}}}
	\end{equation*}
	Together with \(G_{1,k+1}=G_{1,k}\circ G\)
	and \eqref{eq:rra_bounded_bound}
	we obtain for \(z\in\Delta_{\frac{\rho}{\sqrt{n}}}\)
	\begin{align*}
		\abs{\left(G_{1,k+1}(z)\right)_\nu}&=
		\sum_{\substack{\alpha\in\mathbb{N}_0^n\\1\le\abs{\alpha}
		\le d^{n-1}}}\!\!\!\!\!\!
		a_{k,\nu,\alpha}\left(G(z)\right)_1^{\alpha_1}\dotsm
		\left(G(z)\right)_n^{\alpha_n}\\
		&\le\sum_{i=1}^{d^{n-1}} \frac{\gamma^{k}}{\rho^i}\left(
		\sqrt{n}\right)^iQ(i)C^i
		\rho^i\\
		&\le\gamma^{k+1}.\qedhere
	\end{align*}
\end{proof}

\begin{lemma}\cite{RosayRudin}*{Appendix, Lemma~1~(b)}
\label{th:rra_limit}
	Let \(G\)
	be a lower triangular holomorphic automorphism with diagonal
	elements \(\abs{c_{\nu}}<1\)
	Then we have
	\begin{equation*}
		G_{1,j} \rightrightarrows 0 \text{ on compacts for } j \to \infty.
	\end{equation*}
\end{lemma}

\begin{lemma}
\label{th:inv_est}
	Let \(G\colon \CC^n\to\CC^n\) be a lower triangular holomorphic automorphism
	with
	\(G=A+H\) where
	\(A:=\diff{0}{G}\)
	and \(H\colon \CC^n\to\CC^n\) is a holomorphic self-map.
	We assume that there exist \(0<s<1\), \(\delta>0\) and \(C\ge 0\) with
	\begin{equation*}
		\norm{A^{-1}}\le\frac{1}{s}
	\end{equation*}
	and
	\begin{equation*}
		\forall\,z\in \ball{\delta}: \norm{H(z)}\le C\norm{z}.
	\end{equation*}
	Then
	\begin{equation*}
		\begin{split}
			\forall\,z\in \ball{\frac{\delta s^{n}}{\sqrt{n}\left(n-1\right)!
			\left(1+C\right)^{n-1}}}:
			\norm{G^{-1}(z)}\le \sqrt{n}\left(n-1\right)!
			\frac{\left(1+C\right)^{n-1}}{s^{n}}\norm{z}.
		\end{split}
	\end{equation*}
\end{lemma}	
\begin{proof}
	Let
	\(A^{-1}=\left(A^{-1}_1,\dotsc,A^{-1}_n\right)\) and
	\(H=\left(H_1,\dotsc,H_n\right)\).
	With \(G\) also \(G^{-1}\) is lower triangular. Therefore \(A^{-1}\) is
	lower triangular.
	Hence for \(\nu\in\left\{1,\dotsc,n\right\}\) we have
	\begin{equation}
	\label{eq:rra_lower_inv_est_1}
		\begin{split}
			\forall\,z\in\mathbb{C}^n:\quad
			\abs{A^{-1}_\nu(z)}&=
			\abs{A^{-1}_\nu\left(z_1,\dotsc,z_\nu,0,\dotsc,0\right)}\\
			&\le
			\norm{A^{-1}\left(z_1,\dotsc,z_\nu,0,\dotsc,0\right)}\\&
			\le
			\frac{1}{s}\norm{\left(z_1,\dotsc,z_\nu,0,\dotsc,0
			\right)}.
		\end{split}
	\end{equation}
	\(G\) is lower triangular and therefore
	\(H=G-\diff{0}{G}\) is a lower triangular map with vanishing diagonal 
	elements.
	Hence for \(\nu\in\left\{1,\dotsc,n\right\}\) we have
	\begin{equation}
	\label{eq:rra_lower_inv_est_2}
		\begin{split}
			\forall\,z\in B_\delta:\quad
			\abs{H_\nu(z)}&=
			\abs{H_\nu\left(z_1,\dotsc,z_{\nu-1},0,\dotsc,0\right)}\\
			&\le
			\norm{H\left(z_1,\dotsc,z_{\nu-1},0,\dotsc,0\right)}\\&
			\le
			C\norm{\left(z_1,\dotsc,z_{\nu-1},0,\dotsc,0
			\right)}.
		\end{split}
	\end{equation}
	
	Let \(z\in\mathbb{C}^n\) with
	\(\norm{z}\le \frac{\delta s^{n}}{\sqrt{n}\left(n-1\right)!
	\left(1+C\right)^{n-1}}\). For \(z=G(w)=A(w)+H(w)\) we have
	\(G^{-1}(z)=w=A^{-1}\left(z-H(w)\right)\).
	By \eqref{eq:rra_lower_inv_est_1} and \eqref{eq:rra_lower_inv_est_2}
	it follows that
	\begin{equation*}
		\abs{w_1}\le\frac{1}{s}\abs{z_1}+0=(1-1)!\frac{\left(
		1+C\right)^{1-1}}{s^1}
		\sqrt{\sum_{\eta=1}^{1}\abs{z_\eta}^2}.
	\end{equation*}
	For \(\nu\in\left\{2,\dotsc,n\right\}\) assume that
	\begin{equation*}
		\abs{w_\mu}\le(\mu-1)!\frac{\left(
		1+C\right)^{\mu-1}}{s^\mu}
		\sqrt{\sum_{\eta=1}^{\mu}\abs{z_\eta}^2}
	\end{equation*}
	for all \(\mu\in\left\{1,\dotsc,\nu-1\right\}\).
	
	Together with
	\eqref{eq:rra_lower_inv_est_1} and
	\eqref{eq:rra_lower_inv_est_2} and noting that \(s<1\) and \(1+C\ge 1\)
	we finally obtain the following estimates: \pagebreak
	\begin{align*}
		\abs{w_\nu} &\le \frac{1}{s}\left(\sqrt{\sum_{\eta=1}^\nu
		\abs{z_\eta}^2}+\sqrt{\sum_{\eta=1}^\nu \abs{H_\eta(w)}^2}\,\right)\\
		&\le \frac{1}{s}\left(\sqrt{\sum_{\eta=1}^\nu
		\abs{z_\eta}^2}+C\sqrt{\nu-1}
		\sqrt{\sum_{\eta=1}^{\nu-1} \abs{w_\eta}^2}\,\right)\\
		&\le\frac{1}{s}\left(\sqrt{\sum_{\eta=1}^\nu
		\abs{z_\eta}^2}+C\left(\nu-1\right)
		\sqrt{
		\!\left(\left(\nu-2\right)!\frac{\left(1+C\right)^{\nu-2}}{s^{\nu-1}}
		\right)^{\!\!2}\sum_{\eta=1}^{\nu-1} \abs{z_\eta}^2}\,\right)\\
		&\le\frac{1}{s}\left(\sqrt{\sum_{\eta=1}^\nu
		\abs{z_\eta}^2}+C\left(\nu-1\right)!
		\frac{\left(1+C\right)^{\nu-2}}{s^{\nu-1}}
		\sqrt{
		\sum_{\eta=1}^{\nu} \abs{z_\eta}^2}\,\right)\\
		&\le \left(\nu-1\right)!
		\frac{\left(1+C\right)^{\nu-2}}{s^{\nu}}\left(
		\left(1+C\right)\sqrt{\sum_{\eta=1}^{\nu} \abs{z_\eta}^2}\,\right)\\
		&=\left(\nu-1\right)!
		\frac{\left(1+C\right)^{\nu-1}}{s^{\nu}}
		\sqrt{\sum_{\eta=1}^{\nu} \abs{z_\eta}^2}
	\end{align*}
	In particular
	we have shown the desired estimate by induction.
\end{proof}

\begin{proof}[Proof of Theorem \ref{th:main}]
	We use Proposition \ref{pp:rr_biholo}.
	W.l.o.g. let \(z_0=0\).
	
	It is easy to see that every eigenvalue \(\lambda\) of \(\diff{0}{f_{j}}\)
	satisfies
	\begin{equation}
	\label{eq:rra_rra_rra_ew}
		0<s\le \abs{\lambda}\le r<1.
	\end{equation}
	For \(j\in\NN\) Lemma \ref{th:rra_3}
	gives us a unitary linear map \(S_{j}\), a lower triangular
	automorphism
	\(\widetilde{G}_{j}\) and holomorphic maps \(T_{j}^m\) with
	\begin{equation}
	\label{eq:rra_rra_rra_maps_est}
		S_{j}\circ \widetilde{G}_{j}^{-1}\circ S_{j}^{-1}\circ T_{j}^m
		\circ f_{j}-T_{j}^m\in
		\landau{m}.
	\end{equation}
	
	If \(q\ge 2\), we have (according to Lemma \ref{th:rra_3})
	by \eqref{eq:der_equ}
	that
	all \(S_{j}=:S\) are identical.
	For \(j\in\mathbb{N}\) we define
	\begin{equation}
	\label{eq:rra_rra_rra_def}
		G_j := S\circ\widetilde{G}_j\circ S^{-1}
	\end{equation}
	which fulfills \eqref{eq:rra_rra_maps_0} and
	\eqref{eq:rra_rra_maps_diff} of Proposition \ref{pp:rr_biholo}.
	If \(q\ge p\), we have (again, according to Lemma \ref{th:rra_3})
	by
	\eqref{eq:rra_rra_rra_ew} that all
	\(\widetilde{G}_j\) are identical.
	With \(\widetilde{G}_j\) also \(\widetilde{G}_j^{-1}\) is
	lower triangular. Hence (noting that \(S\) is unitary)
	Lemma \ref{th:rra_bounded} gives us \(0<\tilde{\rho}<1\) and
	\(\gamma>0\) s.t.
	\begin{equation*}
		G_{j+1,k}^{-1}\left(\poly{\frac{\rho}{\sqrt{n}}}\right)\subset
		\Delta_{\gamma^{k-j}}
	\end{equation*}
	for \(j>k\).
	Application of the Schwarz--Pick Lemma then gives us some \(\rho > 0\) and
	\(a>0\) s.t. \eqref{eq:rra_rra_maps_sp} holds
	(with \(\gamma>0\) from above).
	By Lemma \ref{th:rra_limit}  and
	\eqref{eq:rra_rra_rra_ew} we get that
	\eqref{eq:rra_rra_maps_conv} is fulfilled.
	If \(\diff{0}{f_{j}}\) is normal, then the application of
	Lemma \ref{th:rra_3}
	gives a linear estimate for \(\widetilde{G}_{j}\) and
	\(\deg \widetilde{G}_j\le p-1\).
	Lemma \ref{th:inv_est} then gives a linear estimate for
	\(\widetilde{G}_{j}^{-1}\) and it is easy to see that
	we have \(\deg \widetilde{G}_{j}^{-1}\le \left(p-1\right)^{n-1}\).
	Hence by the application of
	Lemma \ref{th:rra_bounded} we see that we may choose
	\begin{equation}
	\label{eq:rra_rra_est_q}
		\gamma=\mathop{\mathlarger{\mathlarger{\mathlarger{\sum}}}}_{i=1}^{\left(p-1\right)^{\left(n-1\right)^2}}
\!\!\!\!		
		\left(\multii{i}\sqrt{n}\right)^i
		\left(
			\sqrt{n}\left(n-1\right)!\frac{
				\left(
					1+C
				\right)^{n-1}
			}
			{
				s^{n}
			}
		\right)^i
	\end{equation}
	where
	\begin{equation}
	\label{eq:rra_rra_rra_est_C}
		C=r\sum_{m=2}^{p-1}n^2\multiisq{m}\left(
			\frac{\sqrt{n}}{\min\left\{1,\delta\right\}}
		\right)^m.
	\end{equation}
	
	We choose \(q\) (\(\ge p\)) large enough to satisfy
	\(r^q\gamma<1\). Then we define
	\begin{equation*}
		T_j := T_j^q.
	\end{equation*}
	This fulfills \eqref{eq:rra_rra_maps_0} and 
	\eqref{eq:rra_rra_maps_diff} of Proposition \ref{pp:rr_biholo}. By
	\eqref{eq:der_equ} we have (according to Lemma \ref{th:rra_3})
	that all \(T_{j}\) are identical. Hence
	\eqref{eq:rra_rra_maps_bound} and \eqref{eq:rra_rra_maps_inj}
	hold.
	All maps in \eqref{eq:rra_rra_rra_maps_est}
	are uniformly bounded (on \(\ball{\delta}\)) and hence
	\eqref{eq:rra_rra_maps_est} is fulfilled.

	The theorem now follows from Proposition \ref{pp:rr_biholo}.
	If \(\diff{0}{F_{j}}\)
	is normal, the desired estimate
	follows from \(r^q\gamma<1\) by
	\eqref{eq:rra_rra_est_q}.
	If in addition \(n=2\), the desired estimate 
	follows from Remark \ref{th:rem_est}.
	In both cases the estimates satisfy \(q\ge p\ge 2\) which is
	needed above.
\end{proof}

\begin{remark}
\label{rm:onlyspecial}
The goal in the proof above is to make sure all \(G_j\) and all
\(T_j\) are identical. The assumptions in Theorem \ref{th:main} are
one way to achieve this. There are other possibilities:
\begin{enumerate}
\item If all derivatives of \(f_j\) up to order \(q-1\)
(like in the proof above)
are special elements with respect to
\(\diff{z_0}{f_j}\), Lemma \ref{th:rra_3} assures \(T_j^m=\id\).
Then we need the assumption \eqref{eq:der_equ}
of Theorem \ref{th:main} just for multi-indices up
to order \(p-1\) in order to get \(G_1=G_2=\dotsb\) (according to Lemma
\ref{th:rra_3}).

\item If we have \(\specialp{m}{A_j}=0\) for all
\(m\in\left\{2,\dotsc,p-1\right\}\) and \(A_j:=\diff{z_0}{f_j}\),
Lemma \ref{th:rra_3} gives us \(G_j=A_j\).
Then we may choose (in the proof above) \(\gamma=\frac{1}{s}\)
and therefore \(q=p\).
\end{enumerate}

\end{remark}

\end{document}